\documentclass[11pt]{amsart}
\usepackage[dvips]{graphicx}
\usepackage{amsmath}
\usepackage{amsthm}

\usepackage{amssymb}
\usepackage{mathrsfs}
\usepackage{latexsym}
\theoremstyle{plain}
\newtheorem{thm}{Theorem}[section]
\newtheorem{lem}[thm]{Lemma}
\newtheorem{cor}[thm]{Corollary}
\newtheorem{prop}[thm]{Proposition}
\theoremstyle{definition}
\newtheorem{rem}[thm]{Remark}
\newtheorem{defn}[thm]{Definition}

\numberwithin{equation}{section} 

\newcommand{\R}{\mathbb{R}}
\newcommand{\C}{\mathbb{C}}

\newcommand{\pa}{\partial}

\newcommand{\A}{\alpha}

\begin{document}
\title[ ]%
      {Well-posedness for the 1D cubic nonlinear Schr\"odinger equation in $L^p$, $p>2$ }
\author[R.Hyakuna]{Ryosuke Hyakuna}
\address[2-32-1 Ogawa-nishimachi, Kodaira city, Tokyo, Japan ]{Polytechnic University of Japan}
\email{107r107r@gmail.com}
\begin{abstract}

In this paper, local well-posedness is shown for the one dimensional cubic nonlinear Schr\"odinger equation in 
$L^p$-spaces for $2<p<4$, which generalizes a classical result for $p=2$ by Y. Tsutsumi and recent work for $1<p<2$ by Y. Zhou. As a consequence, a local theory of solutions is developed for a class of data which decay more slowly than square integrable functions.  Regularity properties of the local solutions in the $L^p$-based Sobolev spaces and Stricharz spaces are also proved.

\end{abstract}
\maketitle
%
%

\section{Introduction}
We consider the Cauchy problem for the one dimensional cubic nonlinear Schr\"odinger equation:
\begin{equation}
iu_t+u_{xx}+|u|^2u=0,\quad u|_{t=0}=\phi.  \label{NLS}
\end{equation}

It is well known that (\ref{NLS}) is locally well posed in $L^2(\R)$.  Here the ``well-posedness'' means the existence of a local solution $u:[0,T]\times \R \to \C,\,T\triangleq T(\|\phi\|_{L^2})>0$ for any $\phi \in L^2$, uniqueness (in a suitable solution space), continuous dependence on data, 
and the persistence property of the solution--i.e., $u \in C([0,T]; L^2)$. This is a classical result by Y. Tsutsumi \cite{yT}.  Our interest in this paper is to extend this standard well-posedness in $L^2$ into $L^p, p\neq 2$ in a natural manner.  As far as the author knows, there are not many studies on
(\ref{NLS}) in this direction.  One reason for that is it is widely believed that (\ref{NLS}) is not locally well posed in $L^p$ if $p\neq 2$.  In fact, as early as the 1960s, it was alreadly proved that the initial value problem for the corresponding linear equation
\begin{equation}
iu_t+u_{xx}=0,\quad u|_{t=0}=\phi, \label{LS}
\end{equation}
is not well posed in $L^p(\R)$ unless $p=2$ (see \cite{Hormander}).  In particular, if $\phi \in L^p,\,p\neq 2$ we cannot expect the
solution $u(t)$ of (\ref{LS}) to belong to $L^p$.  Recently, though, there are several works that treat (\ref{NLS}) in data spaces whose norm are characterized by some kind of $p$-th integrability with $p\neq 2$.  For example,
Gr\"unrock \cite{ Grunrock} proved that (\ref{NLS}) is locally well posed in $\widehat{L^p}$ for $1<p<\infty$, where
\begin{equation*}
\widehat{L^p} \triangleq \{\phi \in \mathcal{S}'(\R)\,| \,\hat{\varphi} \in L^{p'} \}.
\end{equation*}
Notice that by the Hausdorff--Young inequality
\begin{equation*}
L^{p} \subset \widehat{L^p},\quad \text{if} \,\,p\le 2,\qquad \widehat{L^p} \subset L^{p},\quad \text{if} \,\,p\ge 2.
\end{equation*}
Thus (\ref{NLS}) can be well posed in spaces which are similar to the $L^p$-spaces even if $p\neq 2$.  Moreover, one 
remarkable work was done by Zhou.  In \cite{Zhou} he consider the corresponding integral equation
\begin{equation}
u(t)=U(t)\phi +i\int^t_0 U(t-s)|u(s)|^2 u(s) ds, \label{introint}
\end{equation}
where $U(t)$ denotes the free Schr\"odinger group, and introduced the twisted variable $v(t)=U(-t)u(t)$ to rewrite (\ref{introint}) as
\begin{equation}
v(t)=\phi+i\int^t_0 U(-s)\left[|U(s)v(s)|^2 (U(s)v(s)) \right] ds. \label{introintv}
\end{equation}
Then he discovered that (\ref{introint}) is locally well posed in $L^p,\,1<p<2$ in the sense that a local solution
$v\in C([0,T(\|\phi\|_{L^p})] ;L^p)$ of (\ref{introintv}) exists for any data $\phi \in L^p$ with uniqueness and continuous depedence on data.  This result suggests that (\ref{NLS}) can be locally well posed in $L^p,\,p\neq 2$ if one consider the ``twisted" persistence property $U(-t)u(t) \in C([0,T] ;L^p)$ of the solution in place of the usual persistency.  Note that by the unitarity property $u(t) \in C(I; L^2)$ if and only if $U(-t)u(t) \in C(I; L^2)$ for any $I\subset \R$.  Thus this kind of well-posedness in $L^p$ can be regarded as a natural extension of the one in $L^2$ in the usual sense.  Now one question arises: does the similar well-posedness hold true for $p>2$?  As far as the author knows there is no previous work that addresses this problem.  But $L^p$-spaces for $p>2$ are most typical spaces whose functions decay at $|x|\to \infty$ more slowly than square integrable fuctions, and we believe it is very interesting to establish a theory of solutions to (\ref{NLS}) for such a class of data.  As a study in a similar direction, we refer to the most recent work \cite{DSS,scjfa,arxiv} where the authors consider Bessel potential spaces $H^{s,p}$ for $p>2$ with $s>0$.  Here in this paper we consider mere $L^p$-spaces -- without any smoothness or other additional assumptions.  Throughout the paper we use the following notations.  For $r\in [1,\infty]$, $r'$ denotes the conjugate of $r$: $1/r+1/r'=1$.  $\mathcal{F}$ and $\mathcal{F}^{-1}$ denote the Fourier transform and the inverse Fourier transform respectively.  We also use $\widehat{\phi}$ to denote the Fourier transform of $\phi$.  $c,C$ are
positive constants that may vary line to line.  In particular, we use $C_{A,B,\cdots}$ when we want to emphasize that
the constant depends on the parameters $A,B,\cdots$.  Let $I \subset \R$ be an interval and let $X$ be a Banach space of complex valued fuctions on $\R$.  The Bessel potential space is denoted by $H^{s,p}$.  We define the space $C_{\mathfrak{S}}(I ;X)$ by
\begin{equation*}
C_{\mathfrak{S}}(I ;X)\triangleq \{ u: I \times \R \to \C\,| \,U(-t)u(t) \in C(I ;X)\,\}.
\end{equation*}

Now we give the definition of local well-posedness introduced by Zhou\cite{Zhou}.

\begin{defn}\label{TLWP}
Let $X$ be a Banach space of $\C$-valued functions on $\R$. We say that (\ref{NLS}) is locally well posed in $X$ if, for any $\phi \in X$ there are $T>0$ and a unique solution $u:[0,T] \times \R \to\C$ of (\ref{NLS}) such that
\begin{equation*}
u \in C_{\mathfrak{S}}([0,T] ; X)\cap Z_T \triangleq S_T,
\end{equation*}
where $Z_T$ is some auxillary space.  Moreover, the map $\phi \mapsto u$ is locally Lipschitz from $L^p$ to $S_T$.

\end{defn}

The main result of the present paper is a local well-posedness in $L^p$ with $p>2$ in the above sense.
\begin{thm} \label{LWPthm}
For $2\le p<4$ Cauchy problem {\rm (\ref{NLS})} is locally well posed in $L^p(\R)$ in the sense of Definition \ref{TLWP}.

\end{thm}

\begin{rem}
It is obvious that if $X$ has the unitarity property i.e. $\|U(t)\phi \|_X=\|\phi\|_X,\,\forall t \in I$, then $C_{\mathfrak{S}}(I; X)=C(I;X)$.  In such cases, local well-posedness in the sense of Definition \ref{TLWP} is equivalent to the usual one equipped with the persistence property of solution.  In particular, the classical result \cite{yT} says (\ref{NLS}) is locally well posed in $L^2$ in the sense of Definition \ref{TLWP}.

\end{rem}

 We prove Theorem \ref{LWPthm} by constructing a local solution
in a space $Y^p_{q,\theta}(T)$ of functions defined on $[0,T]\times \R$ such that 
\begin{equation*}
Y^p_{q, \theta}(T) \hookrightarrow C_{\mathfrak{S}}([0,T] ; L^p).
\end{equation*}
The definition of $Y^p_{q,\theta}$ is given at the end of this section.  Below we state our local well-posedness result in a more precise manner.  For $M>0$ and $1\le p \le \infty$ we set
\begin{equation}
\mathscr{B}^p_M \triangleq \{ \phi \in L^p (\R)\, |\, \|\phi \|_{L^p} \le M\,\}.
\end{equation}

\begin{prop} \label{LWProp}
Let $2\le p <4$.  For any $M>0$ there is a $T_M>0$ with $\lim_{M\to \infty}T_M=\infty$ such that: for any $\phi \in \mathscr{B}^p_M$ there is a 
unique solution $u \in Y_{q,\theta }^p(T_M)$ of (\ref{NLS}).  Moreover, the map $\phi \mapsto u$ is Lipschitz from $\mathscr{B}^p_M$ to $Y_{q,\theta}^p(T_M)$.

\end{prop}

\bigskip

We also present miscellaneous results on the solution given by Proposition \ref{LWProp}.  It would be of interest to pursue the regularity of the solution at the $L^p$-level.  Unfortunately, we cannot expect the solution to be in $L^p$ in general.  This is not unexpected from the well known ill-posed result for the corresponding linear equation.  For the regularity in the $L^p$-framework we have the following results: 
let $2\le p <4$ and $M>0$.  For $\phi \in \mathscr{B}^p_M$ we let $\mathcal{S}_M\phi \in Y^p_{q,\theta}(T_M)$ be the (unique) solution to (\ref{NLS}) given by Proposition \ref{LWProp}.
\begin{cor} \label{regularityp}
Let $2\le p <4$ and $s<-(1-2/p)$.  Then (\ref{NLS}) is locally well posed from $L^p$ to $H^{s,p}$ in the following sense: for any
$M>0$ one has
\begin{enumerate}
\item
$\mathcal{S}_M \phi \in C([0,T_M] \,; \,H^{s,p}(\R))$ for every $\phi \in \mathscr{B}^p_M$.
\item
There exists $C_M>0$ such that
\begin{equation*}
\sup_{t\in [0,T_M]} \|\mathcal{S}_M \phi_1 -\mathcal{S}_M\phi_2 \|_{H^{s,p}(\R)} \le C_M \|\phi_1 -\phi_2\|_{L^p}
\end{equation*}
for any $\phi_1,\phi_2 \in \mathscr{B}^p_M$.

\end{enumerate}

\end{cor}
\medskip

On the other hand, if $s>-(1-2/p)$ the Cauchy problem (\ref{NLS}) is not locally well posed from $L^p$ to $H^{s,p}$
in the above sense.  So in general, from a standpoint of the well-posedness with the $L^p$-based regularity, there is a loss of smoothness of 
order $1-2/p$.  More precisely, we have the following result:

\begin{cor}\label{regularitypn}
Let $2<p\le 3$ and $s>-(1-2/p)$.  Then Cauchy problem (\ref{NLS}) is not locally well posed from $L^p$ to $H^{s,p}$ in the 
sense of Corollary \ref{regularityp}.

\end{cor}

\begin{rem}
The case $3<p<4$ is excluded from the statement of Corollary \ref{regularitypn}.  This is due to the availability of key Strichartz type estimates which lead to
the ill-posedness results.
\end{rem}

Our next interest is the regularity in the so-called Strichartz space.  It is well known that the solution $u: I\times \R\to \C$ of (\ref{NLS}) for $\phi \in L^2$ belongs to the space
$L^{\rho} (I ; L^r(\R))$ for $\rho,r\in [2,\infty]$ satisfying
\begin{equation*}
\frac{2}{\rho}+\frac{1}{r}=\frac{1}{2}.
\end{equation*}
Recall that such a pair of exponents $(\rho,r)$ is called admissible.  Our next regularity result shows that this kind of property for the $L^2$-solution can be extended to the $L^p$-setting for $p>2$ in a very natural manner. 
 For $1\le \rho,r <\infty,\,\alpha \in \R$ and $I\subset [0,\infty)$ we define the space of functions $L^{\rho} (I,\, t^{\A}dt\, ; L^r(\R))$ by
\begin{equation*}
L^{\rho} (I,\, t^{\A}dt\, ; L^r(\R))\triangleq \{ u:I\times \R \to \C\,|\, \| u \|_{ L^{\rho} (I,\, t^{\A}dt\, ; L^r(\R))} <\infty  \},
\end{equation*}
where
\begin{equation*}
\| u \|_{ L^{\rho} (I,\, t^{\A}dt\, ; L^r(\R))} \triangleq \left(  \int_I \left( \int_{\R} |u(t,x)|^r dx \right)^{\frac{\rho}{r}}      t^{\A}  dt\right)^{\frac{1}{\rho}}.
\end{equation*}
We have the following regularity property for the solution $u$ to (\ref{NLS}) given by Theorem \ref{LWPthm}:
\begin{cor}\label{Strreg}
Let $M>0$ and let $2\le p <4$ and $2<\rho,r<\infty$. Assume that
\begin{equation*}
\frac{2}{\rho}+\frac{1}{r}+\frac{1}{p}=1.
\end{equation*}
and $q,r$ satisfy either of (i),(ii) below:
\begin{enumerate}
\item
$\displaystyle{0\le \frac{1}{\rho} <\min \left( \frac{1}{4},\frac{1}{2}-\frac{1}{r}\right)}$,

\item
$\displaystyle{4<r\le \infty}$ and $\displaystyle{\rho=\frac{1}{4}}$.

\end{enumerate}
Then
\begin{equation}
\mathcal{S}_M \phi \in L^{\rho} ([0,T_M],\, t^{\frac{1}{p}-\frac{1}{2}}dt\, ; L^r(\R))
\end{equation}
for any $\phi \in \mathscr{B}^p_M$.

\end{cor}

Now we give the definition of the space $Y_{q,\theta}^p$ along with related spaces.  These spaces were first introduced by Zhou in \cite{Zhou}
to show the well-posedness of (\ref{NLS}) in $L^p,\,1<p<2$.
\begin{defn}
Let $T>0$ and let $1\le p,q < \infty$ and $\theta\in \R$. 
\begin{enumerate}
\item
 The space $\tilde{X}^p_{q,\theta}(T)$ is defined by

\begin{equation*}
\tilde{X}^p_{q,\theta}(T) \triangleq \{v :[0,T]\times \R \to \C\,|\, \|v\|_{\tilde{X}^p_{q,\theta}(T)}<\infty \},
\end{equation*}
where
\begin{equation*}
\|v\|_{\tilde{X}^p_{q,\theta}(T)} \triangleq \left(\int^T_0 \left( s^{\theta} \| (\pa_s v) (s,\cdot)   \|_{L^p}\right)^q ds \right)^{\frac{1}{q}},
\end{equation*}
and $X^p_{q,\theta}(T)$ by
\begin{equation*}
X^p_{q,\theta}(T) \triangleq \{v \in \tilde{X}^p_{q,\theta}(T) |v(0) \in L^p\}
\end{equation*}
equipped with the norm
\begin{equation*}
\|v\|_{X^p_{q,\theta}(T)} \triangleq \|v(0)\|_{L^p}+\|v\|_{\tilde{X}^p_{q,\theta}(T)}.
\end{equation*}

\item
The space $\tilde{Y}^p_{q,\theta}(T)$ is defined by
\begin{equation*}
\tilde{Y}^p_{q,\theta} (T)\triangleq \{ u:[0,T] \times \R \to \C\,| \, U(-t)u(t) \in \tilde{X}^p_{q,\theta}(T) \,\}
\end{equation*}
with
\begin{equation*}
\|u\|_{\tilde{Y}^p_{q,\theta}(T)} \triangleq \| U(-t)u(t)\|_{\tilde{X}^p_{q,\theta}(T)}.
\end{equation*}
The space $Y^p_{q,\theta}(T)$ is defined by
\begin{equation*}
Y^p_{q,\theta}(T)\triangleq \{u:[0,T] \times \R \to \C \,| U(-t)u(t)\in X^p_{q,\theta}(T)  \},
\end{equation*}
equipped with the norm
\begin{equation*}
\|u\|_{Y^p_{q,\theta}(T)} \triangleq \|U(-t)u(t)\|_{X^p_{q,\theta}(T)}.
\end{equation*}
\end{enumerate}
\end{defn}

We easily get the following embedding result:
\begin{lem} {\rm(See e.g. \cite[Lemma 2.1]{107indiana})}
Let $T>0$ and let $1\le p,q < \infty$ and $\theta\in \R$.  Suppose that $-\infty <q'\theta <1$.  Then
\begin{equation*}
X^p_{q,\theta}(T) \hookrightarrow C([0,T] ; L^p (\R)).
\end{equation*}
In particular, the embedding
\begin{equation*}
Y^p_{q,\theta}(T)\hookrightarrow C_{\mathfrak{S}}([0,T]; L^p (\R))
\end{equation*}
holds.
\end{lem}

\section{Key Lemmata}\label{Keylemma}
\subsection{Generalized Strichartz type inequality}
The key to our local well-posedness results is generalized Strichartz type estimates:
\begin{lem} \label{FS}
Let $2\le p <4$.  Then the estimate
\begin{equation}
\|U(t)\phi \|_{L_{xt}^{3p}(\R^2)} \le C\|\hat{\phi}\|_{L^p(\R)} \label{Str}
\end{equation}
holds true.

\end{lem}

In this paper we use estimate (\ref{Str}) in the following form:

\begin{cor}
Let $2\le p <4$.  Then the estimate
\begin{equation}
\|t^{-\frac{2}{3p'}} U(1/4t) \phi \|_{L^{3p'}( \R_{+} ; L^{3p'} (\R))} \le C\|\hat{\phi} \|_{L^p (\R)}
\label{Str2}
\end{equation}
holds true.
\end{cor}

\begin{rem}\label{Strrem}
Here are some comments on the above estimates.
\begin{itemize}
\item
The estimates can be regarded as a generalization of the well-known Strichartz estimate for $\phi \in L^2$ and can be traced back to \cite{Fefferman}.  See also  \cite{CVV} and introduction in \cite{GrunrockKdV}.  The proof can be found in e.g. \cite{107T}.

\item
It is well known that the local in time Strichartz estimate
\begin{equation*}
\|U(t)\phi \|_{L^6([0,T] ;L^6(\R))} \le C\|\phi \|_{L^2(\R)}
\end{equation*}
is exploited to prove the existence of an $L^2$-solution to (\ref{NLS}) for a sufficiently small time $T>0$.  In our analysis we essentially use (\ref{Str}) of the following form to establish a local solution on $[0,T]$:
\begin{equation}
\|U(t)\phi \|_{L_{xt}^{3p}([T^{-1},\infty)\times \R)} \le C\|\hat{\phi}\|_{L^p(\R)} .
\end{equation}
This implies that one needs a global in time version of Strichartz type estimates, even in proving the existence of a ``local'' solution for data $\phi \in L^p,p>2$.

\end{itemize}

\end{rem}

\subsection{Factorization of $U(t)$ and cubic nonlinearity}
In order to construct a solution $u$ of (\ref{NLS}) with twisted persistence property $U(-t)u(t)\in L^p$ we rewrite the cubic nonlinearity in terms of $v(t)\triangleq U(-t)u(t)$.  This is done via factorization of the free Schr\"odinger group $U(-t)$.  This kind of expresssion is known and has been used in \cite{Zhou} and in much earlier studies.  See e.g. \cite{HN}.  Here we present the details for completeness and for  convenience of the reader. To this end we introduce several operators.  The phase modulation operator $M_t$ is defined by
\begin{equation*}
M_t: w \mapsto e^{i\frac{x^2}{4t}}w.
\end{equation*}
The dilation operator $D_t$ is defined by
\begin{equation*}
(D_t w)(x) \triangleq (4\pi it)^{-\frac{1}{2}} w\left( \frac{x}{4 \pi it} \right).
\end{equation*}
$R$ is the reflection operator: $(Rw)(x)\triangleq w(-x)$.  Using these operators we get a factorization of $U(t)$ and $U(-t)$ as follows(see \cite[Chapter 4]{Caz}):
\begin{equation*}
U(t)=M_t D_t \mathcal{F} M_t,\quad U(-t)= M_t^{-1} \mathcal{F}^{-1} D_t^{-1} M_t^{-1}.
\end{equation*}
This leads to our next key lemma:
\begin{lem}\label{factorization}
For $t \neq 0$ the following equality holds:
\begin{equation}
U(-t)[u_1(t)\overline{u_2(t)} u_3(t)]=ct^{-1} M_t^{-1} (M_t U(-t)u_1(t)) \ast (R
\overline{M_t U(-t)u_2(t)})\ast (M_t U(-t) u_3(t)), \label{factorization}
\end{equation}
where $\ast$ denotes the convolution with respect to the space variable and $c$ is an absolute constant.
\end{lem}

\begin{proof}
By the factorization of $U(-t)$, we have
\begin{eqnarray*}
M_tU(-t)u_1(t)\overline{u_2(t)} u_3(t) &=& \mathcal{F}^{-1} D_t^{-1} M_t^{-1} u_1(t)\overline{u_2(t)} u_3(t) \\
&=& \mathcal{F}^{-1} D_t^{-1} (M_t^{-1}u_1(t))(\overline{M_t^{-1}u_2(t)} ) (M_t^{-1} u_3(t) )\\
&=&ct^{-1} \mathcal{F}^{-1}  (D_t^{-1}M_t^{-1}u_1(t))(D_t^{-1}\overline{M_t^{-1}u_2(t)} ) (D_t^{-1}M_t^{-1} u_3(t) )\\
&=&ct^{-1}  (M_t U(-t)u_1(t)) \ast (R
\overline{M_t U(-t)u_2(t)})\ast (M_t U(-t) u_3(t)),
\end{eqnarray*}
where we have also used the following trivial equalities
\begin{equation*}
D_t^{-1}(fgh)=(4\pi it)^{-1} (D_t^{-1} f)(D_t^{-1} g)(D_t^{-1} h),\qquad (\mathcal{F}^{-1} \overline{f})(x) =(R\overline{\mathcal{F}^{-1} f})(x).
\end{equation*}
\end{proof}

\section{Proof of the well-posedness results}
\subsection{Non-linear estimates}
We first prove a key trilinear estimate in $X^p_{q,\theta}$-spaces from which we deduce the desired local well-posedness result via the fixed point theorem.  We introduce the trilinear form $\mathscr{D}(v_1,v_2,v_3)$ by
\begin{equation*}
\mathscr{D}(v_1,v_2,v_3)\triangleq \int^t_0 s^{-1} M_s^{-1} [ (M_sv_1(s)) \ast
(R\overline{M_s v_2(s)} )\ast (M_s v_3(s)) ] ds.
\end{equation*}

\begin{prop}\label{trilinearprop}
Let $T>0$.  Assume that $2\le p <4$.  Then
\begin{equation}
\| \mathscr{D}(v_1,v_2,v_3) \|_{\tilde{X}^p_{q,\theta}(T)} \le C\prod_{j=1}^3 \|v_j \|_{X^p_{1,0}(T)} \label{trilinear}
\end{equation}
with
\begin{equation}
q=\frac{p}{p-1}(=p'),\quad \theta =-(1-\frac{2}{p}). \label{trilinearexponent}
\end{equation}

\end{prop}

\begin{proof}
By the Hausdorff--Young and H\"older inequalities, we have
\begin{eqnarray*}
t^{\theta} \|\pa_t \mathscr{D}(v_1,v_2,v_3) \|_{L^p}
&=& t^{\theta-1} \left\| (M_t v_1(t) )\ast (R\overline{M_t v_2(t)}) \ast (M_t v_3(t)) \right\|_{L^p} \\
&\le & C\prod_{j=1}^3 \left( t^{-\frac{1-\theta}{3}} \| \mathcal{F}^{-1} M_t v_j (t) \|_{L^{3p'}} \right),
\end{eqnarray*}
where we have also used the identity $\mathcal{F}^{-1} R\overline{f}=\overline{\mathcal{F}^{-1} f}$.  Taking $L^q([0,T])$-norm of 
both sides and using H\"older's inequality in the time variable, and the fact that $U(-1/4t)=\mathcal{F}^{-1} M_t \mathcal{F},\,t\neq 0$, we get
\begin{equation*}
\| \mathscr{D}(v_1,v_2,v_3) \|_{\tilde{X}^p_{q,\theta}(T)} 
\le C\prod_{j=1}^3 \| t^{-\frac{1-\theta}{3}} U(-1/4t) \mathcal{F}^{-1} v_j(t) \|_{L^{3q}([0,T] ;L^{3p'})}.
\end{equation*}
We estimate the right hand side.  Observe that for each $j=1,2,3$
\begin{eqnarray*}
 \| t^{-\frac{1-\theta}{3}} U(-1/4t) \mathcal{F}^{-1} v_j(t) \|_{L^{3q}([0,T] ;L^{3p'})}
&=&  \| t^{-\frac{1-\theta}{3}} R\overline{U(-1/4t) \mathcal{F}^{-1} v_j(t) } \|_{L^{3q}([0,T] ;L^{3p'})}\\
&=& \|t^{-\frac{1-\theta}{3}} U(1/4t) \mathcal{F}^{-1} \overline{v_j(t)} \|_{L^{3q}([0,T] ;L^{3p'})}.
\end{eqnarray*}
Now we write
\begin{equation*}
\overline{v_j(t)}=\overline{v_j(0)}+\int^t_0 (\pa_s \overline{ v_j})(s) ds.
\end{equation*}
Using the symbol $\mathscr{U}(t) \triangleq t^{-\frac{1-\theta}{3}}  U(1/4t) \mathcal{F}^{-1}$ we have
\begin{eqnarray*}
\|\mathscr{U}(t) \overline{v_j(t)} \|_{ L^{3q}([0,T] ;L^{3p'})} &\le & \|\mathscr{U}(t)\overline{v_j(0)} \|_{L^{3q}([0,T] ;L^{3p'}) } + \left\|  \int^t_0 \mathscr{U}(t)(\pa_s \overline{v_j})(s) ds            \right\|_{L^{3q}([0,T] ;L^{3p'}) } \\
&\le & \|\mathscr{U}(t)\overline{v_j(0)} \|_{L^{3q}([0,T] ;L^{3p'}) } + \left\|  \int^t_0 \| \mathscr{U}(t)(\pa_s \overline{v_j})(s) \|_{L^{3p'}(\R)} ds            \right\|_{L^{3q}([0,T]) } \\
&\le & \|\mathscr{U}(t)\overline{v_j(0)} \|_{L^{3q}([0,T] ;L^{3p'}) } + \left\|  \int^T_0 \| \mathscr{U}(t)(\pa_s \overline{v_j})(s) \|_{L^{3p'}(\R)} ds            \right\|_{L^{3q}([0,T]) } \\
&\le & \|\mathscr{U}(t)\overline{v_j(0)} \|_{L^{3q}([0,T] ;L^{3p'}) } +  \int^T_0 \| \mathscr{U}(t)(\pa_s \overline{v_j})(s) \|_{L^{3q}([0,T] ;L^{3p'})} ds.     
\end{eqnarray*}
Now we take $q,\theta$ as in (\ref{trilinearexponent}).  Then by (\ref{Str2}) the right hand side of the above inequalities is smaller than
\begin{equation*}
C\left( \|\overline{v_j(0)}\|_{L^p} +\int^T_0 \|\pa_s \overline{v_j(s)} \|_{L^p} ds \right)=
C\left( \|v_j(0)\|_{L^p} +\int^T_0 \|\pa_s v_j(s) \|_{L^p} ds \right) =C\|v_j\|_{X^p_{1,0}(T)}.
\end{equation*}
This proves the nonlinear estimate in question.

\end{proof}

\subsection{Proof of Proposition \ref{LWProp}}  Now we prove the main local well-posedness result.  Consider the integral equation
\begin{equation}
u(t)=U(t)\phi +i \int^t_0 U(t-s) |u(s)|^2 u(s) ds. \label{inteq}
\end{equation}
Let $q,\theta$ be as in (\ref{trilinearexponent}).  Using the nonlinear estimate (\ref{trilinear}), we want to find a fixed point of the operator
\begin{equation*}
(\Phi u)(t) \triangleq U(t)\phi +i\int^t_0 U(t-s) |u(s)|^2 u(s) ds
\end{equation*}
in a closed subset of $Y^p_{q,\theta}(T)$ for a suitable $T>0$.  Throughout the proof we use the convention that $v(t)=U(-t)u(t),\,v_j(t)=U(-t)u_j(t)$.  Using these notations and Lemma \ref{factorization}, we have
\begin{equation}
U(-t)\Phi u(t)=\phi +c\mathscr{D}(v,v,v). \label{rewrite}
\end{equation}
For $a>0$ we define $\mathscr{V}(a)$ by
\begin{equation*}
\mathscr{V}(a)\triangleq \{u \in Y^p_{q,\theta}(T)\,|\, u(0)=\phi,\,\|u\|_{\tilde{Y}^p_{q,\theta}(T)} \le a \}
\end{equation*}
equipped with the distance
\begin{equation*}
d(u_1,u_2)\triangleq \|u_1-u_2\|_{\tilde{Y}^p_{q,\theta}(T)}.
\end{equation*}
We first estimate $\Phi u$ for $u\in \mathscr{V}(a)$.  By (\ref{rewrite}) and (\ref{trilinear}), we have
\begin{eqnarray*}
\|\Phi u \|_{\tilde{Y}^p_{q,\theta}(T)} &=& \|U(-t)\Phi u \|_{\tilde{X}^p_{q,\theta}(T)} \\
&=& c \|\mathscr{D}(v,v,v) \|_{\tilde{X}^p_{q,\theta}(T)} \\
&\le & C\|v\|^3_{X^p_{1,0}(T)}\\
&=&C\|u \|^3_{Y^p_{1,0}(T)}.
\end{eqnarray*}
By H\"older's inequality we get
\begin{eqnarray*}
\|u \|^3_{Y^p_{1,0}(T)} &\le & \left( \|\phi \|_{L^p} +T^{1-\frac{1}{p}} \|u\|_{\tilde{Y}^p_{q,\theta}(T)} \right)^3\\
&\le & 8\|\phi \|_{L^p}^3 +8T^{3(1-\frac{1}{p})} \|u \|^3_{\tilde{Y}^p_{q,\theta}(T)}.
\end{eqnarray*}
Therefore, $\Phi :\mathscr{V}(a)\to \mathscr{V}(a)$ is well defined if we choose $a,T$ so that
\begin{equation}
8C\|\phi \|_{L^p}^3 \le \frac{a}{2},\qquad  8CT^{3(1-\frac{1}{p})} a^3 \le \frac{a}{2}. \label{aT1}
\end{equation}

Similarly, for $u_1,u_2 \in \mathscr{V}(a)$ we have
\begin{eqnarray*}
\|\Phi u_1 -\Phi u_2 \|_{\tilde{Y}^p_{q,\theta}(T)} &=& \|\mathscr{D}(v_1,v_1,v_1)-\mathscr{D}(v_2,v_2,v_2)\|_{\tilde{X}^p_{q,\theta}(T)} \\
&\le &  \|\mathscr{D}(v_1-v_2,v_1,v_1)\|_{\tilde{X}^p_{q,\theta}(T)} + \|\mathscr{D}(v_2,v_1-v_2,v_1)\|_{\tilde{X}^p_{q,\theta}(T)} \\
&& + \|\mathscr{D}(v_2,v_2,v_1-v_2)-\mathscr{D}(v_2,v_2,v_2)\|_{\tilde{X}^p_{q,\theta}(T)} \\
&\le & CT^{1-\frac{1}{p}} \|v_1-v_2\|_{\tilde{X}^p_{q,\theta}(T)} \\
&& \times \sum_{1\le j,k \le 2} 
(\|\phi \|_{L^p} +T^{1-\frac{1}{p}} \|v_j \|_{\tilde{X}^p_{q,\theta}(T)} )
(\|\phi \|_{L^p} +T^{1-\frac{1}{p}} \|v_k \|_{\tilde{X}^p_{q,\theta}(T)} ) \\
&\le & 8C T^{1-\frac{1}{p}} (\|\phi \|_{L^p}^2 +T^{2(1-\frac{1}{p})} a^2 ) \|u_1-u_2\|_{\tilde{Y}^p_{q,\theta}(T)} .
\end{eqnarray*}
Thus $\Phi$ is a contraction mapping if 
\begin{equation}
8C T^{1-\frac{1}{p}} (\|\phi \|_{L^p}^2 +T^{2(1-\frac{1}{p})} a^2 ) <\frac{1}{2}. \label{ccontraction}
\end{equation}

Now we prove the existence of a local solution by the fixed point argument.  Let $M>0$.  For any $\phi \in \mathscr{B}^p_M$ we put
\begin{equation}
a=16CM^3,\quad T=\varepsilon M^{-\frac{2p}{p-1}}(\triangleq T_M), \label{fixedpoint}
\end{equation}
where $\varepsilon>0$ is a constant independent of $a,M$.  Then it is easy to see that $a$ and $T$ defined by (\ref{fixedpoint}) satisfy (\ref{aT1}) and (\ref{ccontraction}) if $\varepsilon$ is sufficiently small.  We choose such an $\varepsilon$.  Then $\Phi:\mathscr{V}(a)\to\mathscr{V}(a)$ is well defined and is a contraction mapping.  By the fixed point theorem, there exists a solution $u\in Y^p_{q,\theta}(T_M)$ of the integral equation (\ref{inteq}).  
Moreover, the uniqueness in $Y^p_{q,\theta}(T_M)$ and continuous dependence on data follows from a similar difference estimate as above.  Consequently, the desired local well-poedness result has been proved.

\section{Proof of the regularity results}
\subsection{Proof of Corollary \ref{regularityp}}
We recall some classical results on the $L^p$-regularity for the solution to the linear Schr\"odinger equation.  Denote $B_{p,r}^s$ by the Besov space of order $s$.  For the definition of the Besov space see e.g. \cite{BL}.
\begin{lem}\cite[6.1 Theorem 1]{Brenner} \label{BrennerL}
Let $1 \le p \le \infty$ and $T>0$.  The estimate
\begin{equation}
\sup_{t\in [0,T]} \|U(t)\phi \|_{L^p(\R)} \le C_T \|\phi \|_{B_{p,1}^s(\R)},\quad \forall \phi \in B_{p,1}^s(\R)
\label{BrennerLe}
\end{equation}
holds true for some $C_T>0$ if and only if $s \ge 2|1/p-1/2|$.

\end{lem}

In particular, by the inclusion relation
\begin{equation*}
H^{s',p}\subset B_{p, 1}^s \subset H^{s,p}
\end{equation*}
for $s<s'$, it follows that the estimate
\begin{equation*}
\sup_{t \in [0,T]} \|U(t)\phi \|_{H^{s,p}(\R)} \le C_T\|\phi \|_{L^p},\quad \forall \phi \in L^p
\end{equation*}
holds true if $s<-(1-2/p)$ and fails if $s>-(1-2/p)$.

\begin{lem} \label{lossconti}
Let $T>0$.  Let $2<p < \infty$ and $s<-(1-2/p)$.  Then $U(t)\phi \in C([0,T] ; H^{s,p}(\R))$ for any $\phi \in L^p(\R)$.

\end{lem}

\begin{proof}
We prove the continuity.  Take $\varepsilon>0$ arbitrarily and fix it.  For $t,t' \in [0,T]$ we write
\begin{equation}
\|U(t)\phi -U(t')\phi \|_{H^{s,p}} \le   \|U(t) (\phi -\tilde{\phi} )\|_{H^{s,p}} +\|U(t)\tilde{\phi} -U(t') \tilde{\phi} \|_{H^{s,p}}
+ \|U(t') (\phi -\tilde{\phi} )\|_{H^{s,p}}  \label{1/3arg}
\end{equation}
for some $\tilde{\phi} \in C_0^{\infty} (\R)$.  By Lemma \ref{BrennerL} and density, we may choose $\tilde{\phi}\in C^{\infty}_0(\R)$ so that
\begin{equation*}
\max \left( \|U(t) (\phi -\tilde{\phi} ) \|_{H^{s,p}}, \|U(t') (\phi -\tilde{\phi} ) \|_{H^{s,p}} \right) \le
C_T \|\phi -\tilde{\phi} \|_{L^p} \le \frac{\varepsilon}{3}.
\end{equation*}
For the second term in the right hand side of (\ref{1/3arg}), we have
\begin{eqnarray*}
\|U(t)\tilde{\phi} -U(t') \tilde{\phi} \|_{H^{s,p}} &\le & \|U(t)\tilde{\phi} -U(t') \tilde{\phi} \|_{L^p} \\
&\le &C \|(e^{it|\cdot|^2}-e^{it'|\cdot|^2}) \mathcal{F} \tilde{\phi} \|_{L^{p'}} \\
&\le & C|t-t'| \left\| |\cdot|^2 \mathcal{F}\tilde{\phi} \right\|_{L^{p'}}.
\end{eqnarray*}
Therefore, there is $\delta=\delta(\varepsilon)>0$ such that
\begin{equation*}
\|U(t)\tilde{\phi} -U(t') \tilde{\phi} \|_{H^{s,p}} \le \frac{\varepsilon}{3}
\end{equation*}
for any $t,t' \in[0,T]$ with $|t-t'|<\delta(\varepsilon)$.  
Now the desired continuity assertion follows from the elementary $\varepsilon/3$-argument.

\end{proof}

\noindent\textit{Proof of Corollary \ref{regularityp}}. \quad It is enough to show the embedding
\begin{equation}
Y_{q,\theta}^p(T) \hookrightarrow C([0,T]\, ; H^{s,p}(\R) ). \label{embedconti}
\end{equation}

To prove this we first check that
\begin{equation}
Y_{q,\theta}^p(T) \hookrightarrow L^{\infty}([0,T]\, ; H^{s,p}(\R) ).  \label{embedlinf}
\end{equation}

Let $u \in Y^p_{q,\theta}(T)$ and write $u(t)=U(t)v(t),\,v(t)=U(-t)u(t)$.  We have
\begin{equation*}
U(t)v(t)=U(t)v(0)+\int^t_0 U(t)(\pa_sv)(\tau) d\tau.
\end{equation*}
Taking $H^{s,p}$-norm of both sides and using Lemma \ref{BrennerL}, we have
\begin{eqnarray*}
\|u(t)\|_{H^{s,p}} &\le & \|U(t)v(0)\|_{H^{s,p}} +\int^t_0 \|U(t)(\pa_{\tau} v)(\tau) \|_{H^{s,p}} d\tau \\
&\le &C_T \|v(0)\|_{L^p} + C_T \int^T_0 \|(\pa_{\tau} v)(\tau) \|_{L^p} d\tau \\
&\le& C_T \|v\|_{X_{1,0}^p(T)} \le C_T  \|v\|_{X_{q,\theta}^p(T)} =\|u\|_{Y_{q,\theta}^p(T)} 
\end{eqnarray*}
for any $t \in [0,T]$.  This proves (\ref{embedlinf}).  Now it is enough to check the continuity of the map $t\mapsto u(t)$ from $[0,T]$ to $H^{s,p}$  
to show (\ref{embedconti}).

For $t,t' \in [0,T]$ we write
\begin{eqnarray*}
u(t)-u(t') &=& U(t) v(0)+\int^t_0 U(t)(\pa_{\tau} v)(s) ds -U(t')v(0) -\int^{t'}_0 U(t')(\pa_{\tau} v)(\tau)d\tau \\
&=& \left[ U(t)v(0)-U(t')v(0)\right] +\left[ \int^t_0\left( U(t)(\pa_{\tau} v)(\tau) -U(t')(\pa_{\tau} v) (\tau) \right) d\tau \right]
+\int^{t'}_t U(t') (\pa_{\tau} v)(\tau) d\tau \\
&=& I_1+I_2+I_3.
\end{eqnarray*}
Now taking $H^{s,p}$-norm and letting $t'$ tend to $t$, we see that $\|I_1\|_{H^{s,p}}$ converges to $0$ by Lemma \ref{lossconti}.  Similarly, 
$I_2$ also tends to $0$ in $H^{s,p}$ since
\begin{equation*}
\|I_2\|_{H^{s,p}} \le \int^T_0 \|   U(t)(\pa_{\tau} v)(\tau) -U(t')(\pa_{\tau} v) (\tau) \|_{H^{s,p}} d\tau
\end{equation*}
and
\begin{equation*}
 \|U(t)(\pa_{\tau} v)(\tau) \|_{L^1_{\tau}([0,T];H^{s,p})} \le C_T\|u\|_{Y_{q,\theta}^p},\quad \forall t \in [0,T]
\end{equation*}
 by Lemma \ref{BrennerL}.  For $I_3$ we have
\begin{eqnarray*}
\|I_3\|_{H^{s,p}} &\le &\int^{t'}_t \| U(t') (\pa_{\tau} v)(\tau)\|_{H^{s,p}} d\tau \\
& \le & C_T \int^{t'}_t \|\pa_{\tau} v(\tau)\|_{L^p} d\tau \\
& \le & C_T |t-t'|^{q'\theta +1} \|v\|_{\tilde{X}^p_{q,\theta}(T)},
\end{eqnarray*}
from which it follows that $I_3$ converges to $0$ in $H^{s,p}$ as $t'\to t$.  Consequently, we see that $(t\mapsto u(t))\in C([0,T]; H^{s,p})$.

\qed

\subsection{Proof of Corollary  \ref{regularitypn}}
We need an off-diagonal generalization of the generalized Strichartz estimate (\ref{Str}).
\begin{lem}{\rm(\cite{107T})} \label{odGFS}
Let $2 \le p <4$ and let $q,r$ be such that
\begin{equation*}
\frac{2}{q}+\frac{1}{r}=\frac{1}{p'}.
\end{equation*}
Moreover, assume either of (i),(ii) below:
\begin{enumerate}
\item
$\displaystyle{0\le \frac{1}{q} <\min \left( \frac{1}{4},\frac{1}{2}-\frac{1}{r}\right)}$.

\item
$\displaystyle{4<r\le \infty}$ and $\displaystyle{q=\frac{1}{4}}$.

\end{enumerate}
Then the estimate
\begin{equation}
\| U(t) \phi \|_{L^q(\R ; L^r(\R)} \le C\|\hat{\phi}\|_{L^p}. \label{ofdStr}
\end{equation}
holds true.  In particular, the estimate
\begin{equation}
\| t^{-\frac{2}{q} }U(1/4t) \phi \|_{L^q(\R^{+} ; L^r(\R))} \le C\|\hat{\phi}\|_{L^p}.
\end{equation}
holds true.
\end{lem}

The key to the ill-posedness result is an ``$L^2$-smoothing" for the Duhamel contribution of the solution.
\begin{lem} \label{L2smoothing}
Let $2<p\le 3$ and let $M>0$.  Then
\begin{equation*}
\mathcal{S}_M \phi -U(t)\phi \in C([0,T_M] \,; L^2(\R))
\end{equation*}
 for any $\phi \in \mathscr{B}^p_M$.  In particular,
\begin{equation*}
\sup_{t\in [0,T_M]} \|\mathcal{S}_M \phi -U(t)\phi \|_{L^2} \le CM^{\frac{1}{p-1}}.
\end{equation*}

\end{lem}

\begin{proof}
Let $q_0 \ge 1$.  Observe first that
\begin{equation*}
Y^2_{q_0,0}(T_M) \hookrightarrow C([0,T_M] ;L^2(\R))
\end{equation*}
 by Plancherel's identity.  So we estimate $Y^2_{q_0,0}$-norm of the Duhamel terms of the solution. Arguing as in the proof of Proposition \ref{trilinearprop} we have
\begin{eqnarray*}
\|\mathcal{S}_M \phi -U(t)\phi \|_{Y^2_{q_0,0}(T_M)} &=&\|\mathcal{S}_M \phi -U(t)\phi \|_{\tilde{Y}^2_{q_0,0}(T_M)} 
=\left\| \int^t_0 U(t-s)|u(s)|^2u(s) ds   \right\|_{\tilde{Y}^2_{q_0,0}(T_M)} \\
&=&c\|\mathscr{D}(v,v,v) \|_{\tilde{X}^2_{q_0,0}(T_M)}\\
&\le & C\prod_{j=1}^3 \| t^{-\frac{1}{3}} U(-1/4t) \mathcal{F}^{-1} v(t) \|_{L^{3q_0}([0,T_M];L^6)}.
\end{eqnarray*}
Now we put $q_0=\frac{4p}{5p-6}$.  Then $(q,r)=(3q_0, 6)$ satisfies either (i) or (ii) in the statement of Lemma \ref{odGFS} as long as $2 \le p \le 3$.  Thus arguing as in the proof of Proposition \ref{trilinearprop} we have
\begin{eqnarray*}
\| t^{-\frac{1}{3}} U(-1/4t) \mathcal{F}^{-1} v(t) \|_{L^{3q_0}([0,T_M];L^6)}
&\le& \| t^{\frac{2}{3q_0}-\frac{1}{3}} \|_{L^{\infty}([0,T_M]) }\| t^{-\frac{2}{3q_0}} U(t^{-1}) \mathcal{F}^{-1} v(t) \|_{L^{3q_0}([0,T_M];L^6)} \\
&\le & C T_M^{\frac{2-q_0}{3q_0}} \| v\|_{L^p}.
\end{eqnarray*}
Consequently, for $T_M$ and $a$ as in (\ref{fixedpoint}) we have
\begin{equation*}
\|\mathcal{S}_M \phi -U(t)\phi \|_{Y^2_{q_0,0}(T_M)} \le CT_M^{\frac{2-q_0}{q_0}} \|u\|_{Y^p_{1,0}(T_M)}^3
\le C T_M^{\frac{p-2}{2p}} a=CM^{\frac{1}{p-1}}.
\end{equation*}
\end{proof}

\medskip
\noindent We present two proofs of Corollary \ref{regularitypn}.

\noindent\textit{First proof of Corollary \ref{regularitypn}.}  Assume $s>-(2/p-1)$.  Let $M>0$.  We show that there is a sequence of data $(\phi_n)_n
\subset \mathscr{B}^p_M(\R)\cap \mathcal{S}(\R)$ such that $\lim_{n\to \infty}\sup_{t\in[0,T_M]}\|\mathcal{S}_M \phi_n \|_{L^p}=\infty$.  We may assume that  $s$ is sufficiently close to $-(1-2/p)$, say $s<-(1/2-1/p)$, so that
Sobolev's embedding
\begin{equation*}
L^2(\R) \hookrightarrow H^{s,p}(\R)
\end{equation*}
holds.  Let $M>0$, then by the ``only if" part of Lemma \ref{BrennerL}, we can take a sequence $(\phi_n)_n \subset \mathscr{B}^p_M\cap \mathcal{S}(\R)$ such that
\begin{equation*}
\sup_{t\in[0,T_M]} \|U(t)\phi_n \|_{H^{s,p}} >n.
\end{equation*}
Otherwise one may establish estimate (\ref{BrennerLe}) by density and the Banach--Steinhaus theorem.
Now we write $U(t)\phi_n=\mathcal{S}_M\phi_n+(U(t)\phi_n-\mathcal{S}_M\phi_n)$ and apply Lemma \ref{L2smoothing} to obtain
\begin{eqnarray*}
\sup_{t\in [0,T_M]} \|U(t) \phi_n \|_{H^{s,p}} 
&\le & \sup_{t\in [0,T_M]} \|\mathcal{S}_M \phi_n \|_{H^{s,p}} +\sup_{t\in [0,T_M]} \|\mathcal{S}_M\phi_n -U(t)\phi_n \|_{H^{s,p}} \\
&\le & \sup_{t\in [0,T_M]} \|\mathcal{S}_M \phi_n \|_{H^{s,p}} +\sup_{t\in [0,T_M]} \|\mathcal{S}_M\phi_n -U(t)\phi_n \|_{L^2} \\
&\le & \sup_{t\in [0,T_M]} \|\mathcal{S}_M \phi_n \|_{H^{s,p}} +CM^{\frac{1}{p-1}}.
\end{eqnarray*}
Letting $n\to \infty$ we see that 
\begin{equation*}
\lim_{n\to\infty}\sup_{t\in [0,T_M]} \|\mathcal{S}_M \phi_n \|_{H^{s,p}}=\infty.  
\end{equation*}
This implies that property (ii) in Corollay \ref{regularityp} does not hold.

\qed

\medskip

\noindent\textit{Second proof of Corollary \ref{regularitypn}}.\,\,Let $M>0$.  In the second proof we show that assertion (i) in the statement of Corollary \ref{regularityp} fails by showing the existence of
data $\phi \in \mathscr{B}_M^p$ such that $(\mathcal{S}_M\phi) (t_0) \notin H^{s,p}$ for some $t_0\in (0,T_M)$ if $2<p\le 4$ and $s>2/p-1$.  We first recall the following result on the $L^p$-regularity for the homogeneous data:
\begin{lem}\cite[Theorem 2.6.1]{Caz}
Let $0<a<1$ and $\psi_a(x)=|x|^{-a},\,x\in\R$.  Then
\begin{equation*}
U(t)\psi_a \in L^p(\R)
\end{equation*}
for all $t>0$ and for any $p$ such that
\begin{equation}
p>\max \left( \frac{1}{a},\frac{1}{1-a} \right).  \label{parange}
\end{equation}

\end{lem}

\bigskip

Observe that when $p>2$ we may choose $a$ such that
\begin{equation}
\frac{1}{p}<a<1-\frac{1}{p}. \label{aprange}
\end{equation}
We take such an $a$ and for $t_0\in (0,T_M)$ and we set
\begin{equation*}
\phi_a\triangleq cU(-t_0)\psi_a=\overline{cU(t_0)\psi_a},
\end{equation*}
which belongs to $L^p$ for any $p$ satisfying (\ref{parange}).  The contant $c$ can be varied depending on the size of $M$.  Here we assume $c=1$ for simplicity.  Clearly, $U(t_0)\phi =\psi_a=|x|^{-a}$.  We show that there is an $a$ such that $(\mathcal{S}_M \phi_a)(t_0) \notin H^{s,p}$ if 
$s>2/p-1$ and $2<p\le 3$.  As in the first proof, we may assume  that $s<1/p-1/2$.  Then by the $L^2$-smoothing, $(\mathcal{S}_M \phi_a)(t_0) -U(t_0)\phi_a \in L^2\subset H^{s,p}$ and it is enough to check that $U(t_0)\phi_a \notin H^{s,p}$ to conclude that $(\mathcal{S}_M \phi_a)(t_0) \notin H^{s,p}$.  We estimate $\langle D \rangle^s U(t_0)\phi_a \triangleq \mathcal{F}^{-1} \langle \cdot \rangle^s \widehat{U(t_0)\phi_a}$, where $\langle \xi \rangle^s\triangleq (1+|\xi|^2 )^{1/2}$.   It is known (\cite[Proposition 1.2.5]{Grafakos}) that $\mathcal{F}^{-1}\langle \cdot \rangle^s$ is strictly positive and satisfies
\begin{equation*}
[\mathcal{F}^{-1}\langle \cdot \rangle^s](x) \ge C |x|^{-s-1}+O(|x|^{-s+1})
\end{equation*}
for $|x|\le 2$.  In particular, for a sufficiently small $\delta>0$ one has
\begin{equation*}
[\mathcal{F}^{-1}\langle \cdot \rangle^s](x)\ge c|x|^{-s-1}
\end{equation*}
for all $x \in I_{\delta} \triangleq [-\delta, \delta]$.  Hence we have
\begin{eqnarray*}
[\langle D \rangle^s U(t_0)\phi_a ](x)&\ge & C\int_{y\in x+I_{\delta}} |x-y|^{-1-s}|y|^{-a} dy+C\int_{y\notin x+I_{\delta}} [\mathcal{F}^{-1}\langle \cdot \rangle^s](x-y) |y|^{-a} dy\\
&\ge& C\int_{y\in x+I_{\delta}} |x-y|^{-1-s}|y|^{-a} dy.
\end{eqnarray*}
Now we can easily verify that the function in the right hand side is not in $L^p([-\delta/4,\delta/4])$ as follows.  We may write
\begin{eqnarray*}
C|x|^{-s-a}&=&(|\cdot|^{-1-s}\ast|\cdot|^{-a} )(x) \\
&=&\int_{y\in x+I_{\delta}} |x-y|^{-1-s}|y|^{-a} dy+\int_{y\notin x+I_{\delta}} |x-y|^{-1-s}|y|^{-a} dy\\
&\triangleq &H_1(x)+H_2(x),
\end{eqnarray*}
where we have used some basic facts (see e.g. \cite{GrafakosC}) on the convolution and the Fourier transform of the homogeneous functions.  Clearly, $|\cdot|^{-s-a} =H_1+H_2\notin L^p_{loc}(\R)$ if $a\ge 1/p-s$.  In view of (\ref{aprange}), such an $a$ exists if $1/p-s<1-1/p$, which is equivalent to $s>2/p-1$. On the other hand, we see that $|x|<\delta/4$ and $|x-y|>\delta$ implies $|y|\ge (3/4)\delta$ and $(2/3)|y|\le |x-y|$.  Thus we have
\begin{equation*}
H_2(x) \le (2/3)^{-1-s}\int_{|y|>(3/4)\delta} |y|^{-1-s-a} dy =C_{\delta}<\infty
\end{equation*}
for any $x\in [-\delta/4,\delta/4]$.  Hence $H_1 \notin L^p(\R)$ and consequently, we see that $U(t_0)\phi_a \notin H^{s,p}$.
\qed

\bigskip

\subsection{ Proof of Corollary \ref{Strreg}.} \quad Finally, we prove the result on the Strichartz regularity for the local solution. It suffices to prove that the embedding
\begin{equation}
Y_{q,\theta}^p(T_M) \hookrightarrow L^{\rho}([0,T_M], t^{\frac{1}{p}-\frac{1}{2}}dt ; L^r(\R)) \label{embedstr}
\end{equation}
holds.  To prove this inclusion relation it is enough to show the following Strichartz type estimate:
\begin{equation}
\|U(t)\phi \|_{L^q(\R,t^{1/p-1/2}dt ;L^r(\R))} \le C\|\phi \|_{L^p} \label{abstr}.
\end{equation}
Indeed, once (\ref{abstr}) is verified, the desired embedding follows arguing as in the proof of (\ref{embedlinf}).

Essentially, the estimate is equivalent to (\ref{ofdStr}).  We prove (\ref{abstr}) by showing this.  Recall the factorization of $U(t)$ in Section \ref{Keylemma}.  For $f\in S(\R)$ we have
\begin{equation*}
|U(t) f| = |D_t \mathcal{F} M_t \mathcal{F}^{-1} \mathcal{F} f|
=|D_t\overline{ \mathcal{F} M_t \mathcal{F}^{-1} \mathcal{F} f}|=|D_t\mathcal{F}^{-1} \overline{M_t} \mathcal{F} \mathcal{F}^{-1} \overline{f}|
=|D_t U(1/4t) \mathcal{F}^{-1} \overline{f}|.
\end{equation*}
Now we substitute $\phi=\mathcal{F}^{-1} \overline{f}$\,(i.e.\,$f=\mathcal{F}^{-1} \overline{\phi}$) into the above equality and apply (\ref{ofdStr}) to obtain
\begin{equation*}
\|D_tU(1/4t)\phi\|_{L^q(\R; L^r(\R))} =\|U(t) \mathcal{F}^{-1}\overline{\phi} \|_{L^q(\R; L^r(\R))} 
\le C\|\overline{\phi}\|_{L^p}=C\|\phi\|_{L^p}.
\end{equation*}
Finally, a suitable change of the space and time variables in the left hand side of the above inequality
yields (\ref{abstr}).

\qed

\end{document}